\documentclass[preprint,12pt]{elsarticle}
\biboptions{sort&compress}
\usepackage{geometry}
\makeatletter\let\ps@pprintTitle\makeatother
\usepackage{amssymb,tikz}
\usetikzlibrary{calc}
\usepackage{graphicx,url}
\usepackage{amsmath}
\usepackage{booktabs} 
\usepackage{amsthm}
\usepackage{fontawesome5}
\usepackage{cleveref}
\crefformat{equation}{(#2#1#3)}
\newcounter{statement}[section]

\newtheorem{theorem}[statement]{Theorem}
\newtheorem{lemma}[statement]{Lemma}
\newtheorem{proposition}[statement]{Proposition}
\newtheorem{corollary}[statement]{Corollary}
\newtheorem{remark}[statement]{Remark}
\theoremstyle{definition}
\newtheorem{definition}[statement]{Definition}

\usepackage{algorithm}
\usepackage{algpseudocode} 
\numberwithin{equation}{section}

\DeclareMathOperator*{\argmin}{argmin}
\newcommand{\bF}{\mathbf{F}}
\newcommand{\bp}{\mathbf{p}}
\newcommand{\M}{\mathcal{M}}
\newcommand{\dd}{\mathrm{d}}
\newcommand{\bq}{\mathbf{q}}
\newcommand{\bd}{\mathbf{d}}
\newcommand{\Sph}{\mathbb{S}}
\newcommand{\btheta}{\boldsymbol{\theta}}
\newcommand{\bdelta}{\boldsymbol{\delta}}
\newcommand{\R}{\mathbb{R}}
\usepackage{enumitem}
\begin{document}
\begin{frontmatter}
\title{Point Feature Descriptor via Directional\\
  Partition of Unity on Maps}
\author{Phan Thanh An$^{1,3}$ and Dau The Phiet$^{2,3,\text{\faIcon{envelope}}}$
\\[1pt]}

\address{$^1$Institute of Mathematical and Computational Sciences and Faculty of Applied Science, Ho Chi Minh City University of Technology, 268 Ly Thuong Kiet, Dien Hong Ward, Ho Chi Minh City, Vietnam}
\address{$^2$Department of Applied Mathematics, Faculty of Applied Science, Ho Chi Minh City University of Technology, 268 Ly Thuong Kiet, Dien Hong Ward, Ho Chi Minh City, Vietnam}
\address{$^3$Vietnam National University Ho Chi Minh City, Linh Xuan Ward, Ho Chi Minh City,  Vietnam}

\address{$^\text{\faIcon{envelope}}$ Corresponding author}

\begin{abstract}
We develop a functional-analytic framework for smooth directional point descriptors in GPS-free map-based localization.
Given a query point $\mathbf{p}$ in a map $\mathcal{M} \subset \mathbb{R}^d$, the descriptor integrates an environment signal against a partition-of-unity weight family built from a softmax kernel, yielding a $\mathcal{C}^\infty$ alternative to hard angular binning.
Our main contributions are: (i) a totality theorem showing that the associated linear functionals form a total family in $L^2(\mathbb{S}^{d-1})$, establishing asymptotic injectivity of the descriptor map; and (ii) a descriptor-induced seminorm $|f|_{\mathcal{D},n} = \|P_n f\|_{L^2}$, identified via the Gram matrix of the weights, which satisfies a Parseval-type identity $|f|_{\mathcal{D},n} \to \|f\|_{L^2(\mathbb{S}^{d-1})}$ as $n \to \infty$.
Complementary results include Fréchet differentiability, lower semicontinuity under occlusion, and explicit Lipschitz stability bounds with constants depending on the kernel and temperature.
These properties underpin a localization theory in which the descriptor grid enables nearest-neighbor position recovery with a certifiable static error bound, while robot motion generates an observability Gramian whose smallest eigenvalue controls a dynamic error bound and generically resolves symmetry-induced ambiguities that persist under single-observation matching.
\end{abstract}
\end{frontmatter}
\section{Introduction}
\label{sec:introduction}

\subsection{Motivation}

Autonomous navigation without reliance on global positioning systems requires a robot to determine its location solely from onboard sensor measurements and a prebuilt map of the environment (see~\cite{Ouyang2024, An2026} ). 
Central to this capability is the ability to match a local observation -- a laser scan, a depth image, or a point cloud -- against a stored representation of the map.
This matching problem is fundamentally one of \emph{point description}:
given a point $\mathbf{p}$ in the map, one seeks a vector $\mathbf{F}(\mathbf{p}) \in \mathbb{R}^n$ that is both \emph{discriminative} (distinct points yield distinct descriptors) and \emph{stable} (small displacements of $\mathbf{p}$ produce small changes in $\mathbf{F}(\mathbf{p})$).

Classical descriptors such as FPFH~\cite{rusu2009fpfh} and SHOT~\cite{tombari2010shot} encode local geometry by accumulating statistics into fixed angular bins.
While effective in practice, hard-binning strategies suffer from the \emph{bin-boundary effect}: a small perturbation in viewpoint or sensor noise can shift a measurement from one bin to the adjacent one, producing a discontinuous jump in the descriptor. 
This discontinuity breaks gradient-based pipelines and makes rigorous analysis of stability and completeness difficult.

A line of work addresses smoothness by replacing hard bins with soft assignment kernels~\cite{fisher2007,vlad,softshotref}, including constructions based on the von Mises-Fisher (vMF) distribution~\cite{fisher1953,mardia2000}.
However, these works focus primarily on empirical performance; a rigorous functional-analytic treatment of the information content, stability, and localization properties of smooth directional descriptors has, to our knowledge, not been carried out.

\subsection{Key Idea and Distinction from Prior Work}

This paper develops such a treatment. We observe that any smooth directional descriptor of the form
\[
  \phi_i(\mathbf{p}) = \int_{\mathbb{S}^{d-1}}
  f_{\mathbf{p}}(\boldsymbol\theta)w_i(\boldsymbol\theta)
  \mathrm{d}\sigma(\boldsymbol\theta)
\]
is entirely determined by the choice of weight functions $\{w_i\}_{i=1}^n$.
We impose one structural condition on the weights -- the \emph{partition of unity} (PoU):
\[
  \sum_{i=1}^n w_i(\boldsymbol\theta) = 1, \quad w_i \geq 0,
\]
and show that this single constraint unlocks a complete functional-analytic theory.

For the weights we adopt the softmax kernel with temperature $\varepsilon > 0$, whose kernel coincides with the vMF distribution~\cite{fisher1953,mardia2000}:
\[
  w_i(\boldsymbol\theta) =
  \frac{\exp(\mathbf{d}_i\cdot\boldsymbol\theta/\varepsilon)}
   {\sum_j \exp(\mathbf{d}_j\cdot\boldsymbol\theta/\varepsilon)}.
\]
The use of a vMF-type kernel for soft assignment is not new.
What \emph{is} new is the analytical framework built on top of it:
\begin{enumerate}[label=\textup{(\arabic*)}]
  \item[\rm (a)] The PoU constraint implies that $\sum_i \phi_i(\mathbf{p}) = \int f_\mathbf{p}\mathrm{d}\sigma$, the \emph{zeroth moment} of the signal, independently of the choice of basis or temperature.
The descriptor thus decomposes the signal into a basis-free amplitude component and a directional shape component.
  \item[\rm (b)] The descriptor map $f \mapsto \bF(f)$ is a bounded linear operator whose \emph{Gram matrix} 
  $(\mathbf{G}_n)_{ij} = \langle w_i, w_j\rangle_{L^2}$
gives rise to the descriptor-induced seminorm
\[|f|_{\mathcal{D},n} = \|\bF(f)\|_{\mathbf{G}_n^{-1}} = \|P_n f\|_{L^2},\]
where $P_n$ is the $L^2$-orthogonal projection onto $\mathrm{span}\{w_i\}$.
As $n \to \infty$, this seminorm satisfies a \emph{Parseval-type identity}: $|f|_{\mathcal{D},n} \to \|f\|_{L^2}$.
  \item[\rm (c)] The family $\{L_i^{(n)} : f \mapsto \phi_i^{(n)}\}$ is \emph{total} in $L^2(\mathbb{S}^{d-1})$: the descriptor asymptotically separates all distinct signals, providing a functional-analytic foundation for discriminability.
  \item[\rm (d)] The Fréchet differentiability of $\mathbf{F}(\mathbf{p})$ in the map point $\mathbf{p}$ enables the construction of an \emph{observability Gramian} $\mathcal{W}(0,T) = \int_0^T J_\bF^\top J_\bF\mathrm{d}t$ along robot trajectories, linking descriptor geometry to classical observability theory and yielding explicit, certifiable localization error bounds.
\end{enumerate}
To our knowledge, points~(2)--(4) have not previously appeared in the directional descriptor literature.

\subsection{Main Contributions}

\begin{enumerate}[label=\textup{(\roman*)}]
  \item \textbf{Analytical framework.}
We provide a unified functional-analytic treatment of partition-of-unity directional descriptors. Central results include the zeroth-moment identity (energy preservation), Fréchet differentiability with explicit Jacobian formula, lower semicontinuity under occlusion, and Lipschitz stability with constants expressed in terms of the kernel and temperature.

  \item \textbf{Gram norm and Parseval identity.}
We introduce the descriptor-induced seminorm 
$|f|_{\mathcal{D},n} := (\bF(f)^\top \mathbf{G}_n^{-1}\bF(f))^{1/2}$, identify it as the $L^2$-norm of the orthogonal projection of $f$ onto $\mathrm{span}\{w_i\}$, and prove that it converges to $\|f\|_{L^2}$ -- a Parseval-type identity that does not hold for hard-bin descriptors.
This norm also induces a pseudo-metric on the map that is locally bi-Lipschitz equivalent to the Euclidean metric.

  \item \textbf{Asymptotic completeness.}
We prove that the associated family of linear functionals is total in $L^2(\mathbb{S}^{d-1})$: as the directional basis becomes dense and the temperature decreases, the descriptor asymptotically preserves all information in the signal. The proof proceeds via a PoU interpolation argument and is substantially different from standard frame-theoretic completeness proofs.

  \item \textbf{Localization theory.}
We formalize GPS-free localization as a nearest-neighbor problem on a descriptor grid and derive a static error bound in terms of the local injectivity radius $\sigma_{\min}(J_\bF(\bp))$. For a robot moving with known velocity (odometry), we prove a dynamic error bound controlled by the observability Gramian, with explicit conditions under which trajectory information resolves static ambiguities that arise in symmetric environments.
\end{enumerate}

\subsection{Structure of the Paper}
The rest of the paper is structured as follows. \Cref{sec:preliminaries} fixes notation and standing assumptions used throughout. \Cref{sec:formulation} introduces the four-step descriptor construction.
\Cref{sec:pou} defines the softmax PoU weights and establishes their regularity and continuity with respect to the directional basis.
\Cref{sec:properties} develops the analytical properties of the descriptor, organized into three groups: signal-space properties (energy, completeness, Gram norm), stability under basis perturbation (continuity, Lipschitz bounds), and map-space properties (continuity, lower semicontinuity, Fréchet differentiability).
\Cref{sec:localization} develops the localization theory: static error bounds, dynamic Gramian bound, and generic resolution of symmetry.
\section{Notation and Preliminaries}
\label{sec:preliminaries}

Throughout the paper we work in $\R^d$ with $d \in \{2, 3\}$.
\Cref{tab:notation} summarizes the notation used.

\begin{table}[H]
\centering
\caption{Summary of notation.}
\label{tab:notation}
\begin{tabular}{ll}
\toprule
Symbol & Meaning \\
\midrule
$d$ & Ambient dimension ($2$ or $3$) \\
$\M \subset \R^d$ & Map (point cloud or occupancy representation) \\
$\bp \in \M$ & Query point \\
$\Sph^{d-1}$ & Unit sphere in $\R^d$ \\
$\btheta \in \Sph^{d-1}$ & Unit direction vector \\
$\dd\sigma(\btheta)$ & Standard surface measure on $\Sph^{d-1}$ \\
$f : \Sph^{d-1} \to \R$ & Environment signal observed from $\bp$ \\
$n$ & Number of directional bases \\
$\{\bd_i\}_{i=1}^{n}$ & Set of $n$ unit vectors on $\Sph^{d-1}$ \\
$w_i : \Sph^{d-1} \to \R_{\geq 0}$ & $i$-th partition-of-unity weight function \\
$\phi_i(\bp)$ & $i$-th feature component of point $\bp$ \\
$\bF(\bp) \in \R^n$ & Feature vector of point $\bp$ \\
\bottomrule
\end{tabular}
\end{table}
\section{Descriptor Construction}\label{sec:formulation}
This section introduces the proposed descriptor construction. The construction starts from a local observation of the environment
around a map point $\bp\in\M$ and proceeds through three ingredients: an environment signal defined on the unit sphere, a directional basis, and a family of smooth partition-of-unity weights. These elements are then combined to produce a finite-dimensional feature vector that encodes the local geometry surrounding $\bp$.

\begin{definition}[Environment signal]
\label{def:signal}
Let $\bp \in \M$. The \emph{environment signal} $f_{\bp} : \Sph^{d-1} \to \R$ is a scalar function that encodes a measurable property of the map along direction $\btheta$ from $\bp$. Typical instantiations include:
\begin{itemize}[nosep]
  \item \textbf{Range}: $f_{\bp}(\btheta) = $ distance to the nearest surface point in direction $\btheta$;
  \item \textbf{Reflectance intensity}: $f_{\bp}(\btheta) = $ LiDAR return intensity;
  \item \textbf{Surface gradient}: $f_{\bp}(\btheta) = \|\nabla s(\bp + r\btheta)\|$ evaluated at the first surface intersection.
\end{itemize}
When the context is clear we write $f$ instead of $f_{\bp}$.
\end{definition}

\begin{definition}[Directional basis]
\label{def:basis}
A \emph{directional basis} of order $n$ is a set of $n$ unit vectors
\[
  \mathcal{D} = \{\bd_1, \bd_2, \ldots, \bd_n\} \subset \Sph^{d-1}.
\]
The vectors may be arranged uniformly (angular spacing $2\pi / n$ in 2D, or a Fibonacci lattice in 3D) or adapted to the local geometry of $\M$.
\end{definition}

Given an environment signal and a directional basis, the next step is to distribute directional information smoothly among the basis vectors. To this end, we introduce a family of partition-of-unity weights on the sphere (see Definition~\ref{def:pou_weight}), which will replace the hard angular bins used in classical histogram descriptors.

\begin{definition}[Feature component]
\label{def:component}
Given a directional basis $\mathcal{D}$ and the associated PoU weights
$\{w_i\}$. The \emph{$i$-th feature component}
of point $\bp \in \M$ is given by
\begin{equation}
  \phi_i(\bp) =  \int_{\Sph^{d-1}}  f_{\bp}(\btheta)w_i(\btheta)  \dd\sigma(\btheta),  \qquad i = 1, \ldots, n.
  \label{eq:component}
\end{equation}
In two dimensions ($d = 2$), $\Sph^1 \cong [0, 2\pi)$ and \eqref{eq:component} reduces to $\phi_i(\bp) = \int_{0}^{2\pi} f_{\bp}(\theta) w_i(\theta)\dd\theta$.
\end{definition}


\begin{definition}[Point feature vector]
\label{def:feature}
The \emph{feature vector} of point $\bp$ is the concatenation of all feature components:
\begin{equation}
  \boxed{
  \bF(\bp)  = \bigl[\phi_1(\bp),\phi_2(\bp),\ldots,\phi_n(\bp)\bigr]^{\!\top}   \in \R^n.
  }
 \label{eq:feature}
\end{equation}
\end{definition}

\section{Partition-of-Unity Weights}\label{sec:pou}
We define the weight functions $\{w_i\}$ that appear in definition~\ref{def:component} and establish their regularity and continuity with respect to the directional basis.

\begin{definition}[Partition-of-unity weight]
\label{def:pou_weight}
Let $\mathcal{D} = \{\bd_1,\ldots,\bd_n\} \subset \Sph^{d-1}$ be a directional basis and let $\psi \in \mathcal{C}^k([-1,1]),  k\ge 1$ be a non-negative function such that
\begin{equation}
  Z(\btheta) := \sum_{j=1}^{n}\psi\!\left(\btheta^{\!\top}\bd_j\right)  > 0,
  \qquad \forall\btheta\in\Sph^{d-1}.
  \label{eq:denominator}
\end{equation}
The \emph{PoU weight} associated with $\bd_i$ is
\begin{equation}
  w_i(\btheta)  :=  \frac{\psi\!\left(\btheta^{\!\top}\bd_i\right)}{Z(\btheta)},  \qquad \btheta \in \Sph^{d-1}.
  \label{eq:pou_weight}
\end{equation}
Throughout this paper we adopt the \emph{softmax kernel} with temperature $\varepsilon > 0$, whose kernel coincides with that of the von Mises--Fisher distribution on $\Sph^{d-1}$ \cite{fisher1953,mardia2000}:
\begin{equation}
  \psi_\varepsilon(t) := \exp\!\left(\frac{t}{\varepsilon}\right),
  \qquad t\in[-1,1],
  \label{eq:softmax_kernel}
\end{equation}
so that \eqref{eq:pou_weight} becomes explicitly
\begin{equation}
  w_i(\btheta)  =  \frac{\exp\!\left(\dfrac{\btheta ^{\!\top} \bd_i}{\varepsilon}\right)}   {\displaystyle\sum_{j=1}^{n}\exp\!\left(\dfrac{\btheta ^{\!\top}\bd_j}{\varepsilon}\right)}.
  \label{eq:softmax_weight}
\end{equation}
\end{definition}
The parameter $\varepsilon$ controls how strongly each basis direction influences nearby observations. The following remark clarifies the relationship between the proposed soft assignment and the hard directional partition underlying traditional descriptors.
\begin{remark}[Role of the temperature $\varepsilon$]
\label{rem:temperature}
The parameter $\varepsilon > 0$ controls the concentration of each weight around its preferred direction $\bd_i$.
To make this precise, define the \emph{strict dominance region} of $\bd_i$:
\begin{equation}
  A_i :=   \bigl\{\btheta \in \Sph^{d-1}  \big|
 \btheta^{\!\top}\bd_i > \btheta^{\!\top}\bd_j,
  \quad \forall j \neq i\bigr\},
  \label{eq:dominance_region}
\end{equation}
so that $\{A_i\}_{i=1}^n$ is a Borel partition of $\Sph^{d-1}$ up to a set of $\sigma$-measure zero (the boundaries $\btheta^{\!\top}\bd_i = \btheta^{\!\top}\bd_j$ lie on great-circle hyperplanes and are therefore $\sigma$-negligible).
Then, as $\varepsilon \to 0$,
\begin{equation}
  w_i(\btheta) \longrightarrow \mathbf{1}_{A_i}(\btheta)
  \qquad \sigma\text{-a.e.},
  \label{eq:hard_limit}
\end{equation}
and by the Dominated Convergence Theorem \cite{rudin1987} the convergence also holds in $L^p(\Sph^{d-1})$ for every $1 \leq p < \infty$.
The limiting collection $\{\mathbf{1}_{A_i}\}$ is itself a \emph{measurable partition of unity} on $\Sph^{d-1}$ — the measure-theoretic counterpart of the smooth PoU — recovering the hard directional assignment used in classical histogram descriptors. Conversely, as $\varepsilon \to +\infty$, all weights flatten to $1/n$, producing uniform averaging over all directions.
At intermediate values of $\varepsilon$, the softmax PoU interpolates smoothly between these two extremes, which is the regime of interest for the present descriptor.
\end{remark}

\begin{proposition}[PoU and regularity]
\label{prop:pou}
Let $\psi\in C^k([-1,1])$ satisfy \eqref{eq:denominator}. Then the weights $\{w_i\}_{i=1}^n$ defined in \eqref{eq:pou_weight} satisfy:
\begin{enumerate}[label=\textup{(\roman*)}]
  \item $w_i(\btheta) \geq 0$ for all $\btheta \in \Sph^{d-1}$,
  \item $\displaystyle\sum_{i=1}^{n} w_i(\btheta) = 1$
for all $\btheta \in \Sph^{d-1}$, \label{eq:pou}
  \item $w_i \in \mathcal{C}^k(\Sph^{d-1})$.
\end{enumerate}
\end{proposition}

\begin{proof}
(i) is obvious.

(ii) is immediate by summing \eqref{eq:pou_weight} over $i$ and canceling $Z(\btheta) > 0$.

For (iii): the map $\btheta \mapsto \btheta^{\!\top}\bd_i$ is $\mathcal{C}^\infty$ on $\Sph^{d-1}$; by the regularity of $\psi$, the composition $\rho_i(\btheta) := \psi(\btheta^{\!\top}\bd_i)$ is $\mathcal{C}^k$, and so is $Z = \sum_j \rho_j$.

The compactness of $\Sph^{d-1}$ and the continuity of $Z$ implies $Z \geq Z_{\min} := \inf_{\Sph^{d-1}} Z(\btheta) > 0$, the quotient $w_i = \rho_i / Z$ is $\mathcal{C}^k$.
\end{proof}

 
Beyond smoothness with respect to the observation direction $\btheta$, it is important to understand how the weights vary when the directional basis itself is modified. 
This question is particularly relevant when the basis is optimized from data or adapted to the local geometry of the map. 
\medskip
 
We first establish a key lemma on the smoothness of the denominator.
 
\begin{lemma}[Smoothness and positivity of the denominator]
\label{lem:denom}
Suppose $\psi \in \mathcal{C}^k$ and each $\bd_i(\tau)$ is $\mathcal{C}^1$ in $\tau$, with $\mathcal{D}(\tau)$ being $\psi$-covering for all $\tau$ in a neighborhood $U$ of $\tau_0$.
Then:
\begin{enumerate}[label=\textup{(\roman*)}]
  \item $Z(\btheta;\tau) \in \mathcal{C}^1(U)$ for each fixed $\btheta \in \Sph^{d-1}$,
  \item $Z(\btheta;\tau) \geq Z_{\min}(U) > 0$ uniformly in $(\btheta,\tau) \in \Sph^{d-1} \times U$.
\end{enumerate}
\end{lemma}
 
\begin{proof}
(i) Each term $\psi(\btheta^{\!\top}\bd_j(\tau))$ is a $\mathcal{C}^1$ composition of $\psi \in \mathcal{C}^k$ and $\tau \mapsto \btheta^{\!\top}\bd_j(\tau) \in \mathcal{C}^1$; the finite sum inherits this regularity.

(ii) The function $(\btheta,\tau) \mapsto Z(\btheta;\tau)$ is continuous on the compact set $\Sph^{d-1} \times \overline{U}$ and strictly positive by the $\psi$-covering assumption; the infimum is therefore attained and positive.
\end{proof}
 
\begin{proposition}[Continuity of PoU weights w.r.t.\ the directional basis]
\label{prop:continuity_basis}
Under the hypotheses of lemma~\ref{lem:denom}, for every fixed $\btheta \in \Sph^{d-1}$ and every $i = 1, \ldots, n$, the map
\begin{equation}
  \tau \longmapsto w_i(\btheta;\tau)
  \label{eq:map_tau}
\end{equation}
is continuous on $U$.
Moreover, $w_i(\cdot;\tau)$ converges uniformly on $\Sph^{d-1}$ as $\tau \to \tau_0$:
\begin{equation}
  \lim_{\tau \to \tau_0}  \sup_{\btheta \in \Sph^{d-1}}
  \bigl|w_i(\btheta;\tau) - w_i(\btheta;\tau_0)\bigr|  = 0.
  \label{eq:uniform_continuity}
\end{equation}
\end{proposition}
 
\begin{proof}
Fix $i$ and let $\tau \to \tau_0$ in $U$.
Write $\rho_i(\btheta;\tau) := \psi(\btheta^{\!\top}\bd_i(\tau))$.
 
\medskip
\noindent\textbf{Step 1 — Pointwise continuity of $\rho_i$.}

Since $\bd_i(\tau) \to \bd_i(\tau_0)$ and the inner product is continuous, $\btheta^{\!\top}\bd_i(\tau) \to \btheta^{\!\top}\bd_i(\tau_0)$ for every $\btheta$.
By continuity of $\psi$, $\rho_i(\btheta;\tau) \to \rho_i(\btheta;\tau_0)$ pointwise.
 
\medskip
\noindent\textbf{Step 2 — Uniform continuity via compactness.}

The convergence in Step 1 is actually uniform in $\btheta$.
Indeed, for any $\btheta \in \Sph^{d-1}$,
\begin{equation}
  |\rho_i(\btheta;\tau) - \rho_i(\btheta;\tau_0)|  =
  \bigl|\psi\!\left(\btheta^{\!\top}\bd_i(\tau)\right) - \psi\!\left(\btheta^{\!\top}\bd_i(\tau_0)\right)\bigr|.
  \label{eq:rho_diff}
\end{equation}
Let $\omega_\psi$ denote the modulus of continuity of $\psi$ on $[-1,1]$ (which is uniform since $[-1,1]$ is compact).
By Cauchy–Schwarz inequality, $|\btheta^{\!\top}(\bd_i(\tau)-\bd_i(\tau_0))| \leq \|\bd_i(\tau)-\bd_i(\tau_0)\|$, so
\begin{equation}
  \sup_{\btheta \in \Sph^{d-1}}  |\rho_i(\btheta;\tau) - \rho_i(\btheta;\tau_0)|
  \leq  \omega_\psi\!\left(\|\bd_i(\tau)-\bd_i(\tau_0)\|\right)  \xrightarrow{\tau\to\tau_0} 0.
  \label{eq:rho_uniform}
\end{equation}
 
\medskip
\noindent\textbf{Step 3 — Continuity of $w_i = \rho_i / Z$.}
Write
\begin{equation}
  w_i(\btheta;\tau) - w_i(\btheta;\tau_0)  =   \frac{\rho_i(\btheta;\tau)Z(\btheta;\tau_0) - \rho_i(\btheta;\tau_0)Z(\btheta;\tau)}   {Z(\btheta;\tau)Z(\btheta;\tau_0)}.
  \label{eq:diff_quotient}
\end{equation}
Adding and subtracting $\rho_i(\btheta;\tau_0)Z(\btheta;\tau_0)$ in the numerator of \eqref{eq:diff_quotient} gives
\begin{align}
  \bigl|w_i(\btheta;\tau) - w_i(\btheta;\tau_0)\bigr|  &\leq   \frac{|\rho_i(\btheta;\tau)-\rho_i(\btheta;\tau_0)|}{Z(\btheta;\tau)}
  +  \underbrace{\frac{\rho_i(\btheta;\tau_0)}{Z(\btheta;\tau_0)}}_{\leq1}
  \cdot  \frac{|Z(\btheta;\tau)-Z(\btheta;\tau_0)|}{Z(\btheta;\tau)}.
  \label{eq:bound_split}
\end{align}
By lemma~\ref{lem:denom}(ii), $Z(\btheta;\tau) \geq Z_{\min} > 0$ uniformly.
Summing \eqref{eq:rho_uniform} over $j$ yields the uniform bound $\sup_{\btheta\in\Sph^{d-1}} |Z(\btheta;\tau) - Z(\btheta;\tau_0)| \leq n\omega_\psi(\|\bd_j(\tau)-\bd_j(\tau_0)\|_{\max})$.
Inserting into \eqref{eq:bound_split} and taking the supremum over $\btheta$,
\begin{equation}
  \sup_{\btheta \in \Sph^{d-1}}  \bigl|w_i(\btheta;\tau) - w_i(\btheta;\tau_0)\bigr|
  \leq  \frac{(n+1)}{Z_{\min}}  \omega_\psi\!\left(\max_j\|\bd_j(\tau)-\bd_j(\tau_0)\|\right)  \xrightarrow{\tau\to\tau_0} 0,
  \label{eq:final_bound}
\end{equation}
which establishes \eqref{eq:uniform_continuity}. Pointwise continuity \eqref{eq:map_tau} follows as a special case.
\end{proof}
 
\begin{remark}[Lipschitz case and softmax]
\label{rem:lipschitz}
If $\psi$ is Lipschitz with constant $L_\psi$ on $[-1,1]$, then $\omega_\psi(\delta) \leq L_\psi\delta$ and \eqref{eq:final_bound} gives the explicit uniform Lipschitz bound
\[
  \sup_{\btheta}\bigl|w_i(\btheta;\tau) - w_i(\btheta;\tau_0)\bigr|  \leq
  \frac{(n+1)L_\psi}{Z_{\min}}  \max_j\|\bd_j(\tau) - \bd_j(\tau_0)\|.
\]
For the softmax kernel $\psi_\varepsilon(t) = \exp(t/\varepsilon)$, we have $L_\psi = \varepsilon^{-1}\exp(1/\varepsilon)$ and $Z_{\min} \geq n\exp(-1/\varepsilon)$, yielding
\[
  \sup_{\btheta}\bigl|w_i(\btheta;\tau) - w_i(\btheta;\tau_0)\bigr|
  \leq  \frac{(n+1)e^{2/\varepsilon}}{n\varepsilon}
  \max_j\|\bd_j(\tau) - \bd_j(\tau_0)\|.
\]
This constant grows as $\varepsilon \to 0$, reflecting the fact that sharper weights (near the hard-assignment limit) are more sensitive to perturbations of the basis.
\end{remark}
 
\begin{remark}[Differentiability]
\label{rem:differentiability}
If $\psi \in \mathcal{C}^1$ and each $\tau \mapsto \bd_i(\tau)$ is $\mathcal{C}^1$, one can differentiate $w_i(\btheta;\tau)$ with respect
to $\tau$ by the quotient rule:
\[
  \frac{\partial w_i}{\partial \tau_\ell}  =
  \frac{1}{Z}  \left(\frac{\partial \rho_i}{\partial \tau_\ell} - w_i \sum_{j=1}^n \frac{\partial \rho_j}{\partial \tau_\ell}  \right),
\]
where
$\dfrac{\partial \rho_i}{\partial \tau_\ell}  = \psi'\!\left(\btheta^{\!\top}\bd_i\right)   \btheta^{\!\top}\dfrac{\partial \bd_i}{\partial \tau_\ell}$.
This makes gradient-based optimization of the directional basis (with respect to any differentiable objective) well-defined.
\end{remark}

\section{Properties of the Descriptor}
\label{sec:properties}

The properties of $\bF(\bp)$ naturally fall into three groups, presented in the order below.
\emph{Signal-space properties} (Group A) characterize what the descriptor captures about the environment signal $f$: energy preservation, asymptotic completeness of the functional family, and the Parseval-type Gram norm.
\emph{Stability properties} (Group B) quantify how much $\bF$ changes when the directional basis is perturbed.
\emph{Map-space properties} (Group C) describe how $\bF$ varies with the query point $\bp \in \M$.

\subsection*{A. Signal-Space Properties}
\label{sec:signal_space}

\subsection{Energy Preservation}
 
\begin{proposition}[Energy preservation]
\label{prop:energy}
For any $\bp \in \M$ and any $f \in L^1(\Sph^{d-1})$,
\begin{equation}
  \sum_{i=1}^{n} \phi_i(\bp)  =  \int_{\Sph^{d-1}} f(\btheta) \dd\sigma(\btheta).
  \label{eq:energy}
\end{equation}
\end{proposition}
 
\begin{proof}
By linearity of the integral and proposition~\ref{prop:pou}(ii),
\[
  \sum_{i=1}^{n} \phi_i(\bp)  =  \int_{\Sph^{d-1}} f(\btheta)    \underbrace{\sum_{i=1}^{n} w_i(\btheta)}_{=1}
  \dd\sigma(\btheta)  =  \int_{\Sph^{d-1}} f(\btheta) \dd\sigma(\btheta). \qedhere
\]
\end{proof}
 
\begin{remark}
proposition~\ref{prop:energy} shows that the sum of all feature components equals the total ''energy'' of the signal $f$ regardless of the choice of basis $\mathcal{D}$ or temperature $\varepsilon$.
In particular, the PoU construction distributes — but never creates or destroys — signal energy across the $n$ directional channels.
\end{remark}

\subsection{Asymptotic Completeness and Discriminability}
Having established pointwise and stability properties of the descriptor, we now address a more fundamental question: does the family $\{\phi_i^{(n)}\}$ carry \emph{enough} information to distinguish any two points with distinct environment signals? For a fixed basis $\mathcal{D}^{(n)}$ of finite order $n$, the answer can be negative — two signals differing only within the orthogonal complement of $\mathrm{span}\{w_i^{(n)}\}$ in $L^2(\Sph^{d-1})$ produce identical descriptors.
However, as the basis becomes dense on $\Sph^{d-1}$ and the temperature $\varepsilon^{(n)} \to 0$, the weights $\{w_i^{(n)}\}$ approximate the indicators of arbitrarily fine Borel partitions of $\Sph^{d-1}$ (remark~\ref{rem:temperature}), so their linear span becomes dense in $L^2(\Sph^{d-1})$.
The following theorem makes this precise by showing that the associated family of linear functionals is \emph{total}: the only signal annihilated by all of them is zero.
\begin{theorem}[Asymptotic Completeness and Totality]
\label{theo:completeness}
Let $\{\mathcal{D}^{(n)}\}_{n \geq 1}$ be a sequence of directional bases with mesh norm
\[
  h_n := \max_{\btheta \in \Sph^{d-1}}\min_{i}  \angle(\btheta, \bd_i^{(n)}) \longrightarrow 0,
\]
and let the temperature satisfy $\varepsilon^{(n)} \to 0$ as $n \to \infty$.
Then the family of linear functionals
\[
  \left\{L_i^{(n)} : f \mapsto \phi_i^{(n)}  = \textstyle\int_{\Sph^{d-1}} fw_i^{(n)}\dd\sigma\right\}_{i,n}
\]
is \emph{total} in $L^2(\Sph^{d-1})$, that is if $\phi_i^{(n)} = 0$ for all $i = 1,\ldots,n$ and all $n \geq 1$, then $f = 0$ in $L^2(\Sph^{d-1})$.
\end{theorem}

\begin{proof}
Let $f \in L^2(\Sph^{d-1})$ satisfy $\int_{\Sph^{d-1}} fw_i^{(n)}\dd\sigma = 0$ for all $i$ and all $n$.
We will show that $\int_{\Sph^{d-1}} fg\dd\sigma = 0$ for every $g \in \mathcal{C}(\Sph^{d-1})$, which implies $f = 0$ in $L^2$ since $\mathcal{C}(\Sph^{d-1})$ is dense in $L^2(\Sph^{d-1})$.

\medskip
\noindent\textbf{Step 1 — PoU interpolant of $g$.}
Given $g \in \mathcal{C}(\Sph^{d-1})$, define its \emph{partition-of-unity interpolant} at the basis $\mathcal{D}^{(n)}$:
\begin{equation}
  g^{(n)}(\btheta)  :=  \sum_{i=1}^{n} g\!\left(\bd_i^{(n)}\right) w_i^{(n)}(\btheta),
  \qquad \btheta \in \Sph^{d-1}.
  \label{eq:pou_interp}
\end{equation}
By the partition-of-unity property $\sum_i w_i^{(n)} = 1$,
\begin{equation}
  \bigl|g^{(n)}(\btheta) - g(\btheta)\bigr|  =  \Bigl|\sum_{i=1}^n
\bigl(g(\bd_i^{(n)}) - g(\btheta)\bigr) w_i^{(n)}(\btheta)  \Bigr|
  \leq  \sum_{i=1}^n \bigl|g(\bd_i^{(n)}) - g(\btheta)\bigr| w_i^{(n)}(\btheta).
  \label{eq:interp_err}
\end{equation}

\medskip
\noindent\textbf{Step 2 — Uniform convergence $g^{(n)} \to g$.}
Fix $\eta > 0$.
Since $g$ is uniformly continuous on the compact set $\Sph^{d-1}$, there exists $\delta = \delta(\eta) > 0$ such that $|g(\btheta') - g(\btheta)| \leq \eta$ whenever $\angle(\btheta',\btheta) \leq \delta$.
Split the sum in \eqref{eq:interp_err} into near and far indices:
\begin{equation}
  \bigl|g^{(n)}(\btheta) - g(\btheta)\bigr|
  \leq  \eta \underbrace{\sum_{\angle(\bd_i^{(n)},\btheta)\leq\delta} w_i^{(n)}(\btheta)}_{\leq1}
  +  2\|g\|_\infty  \underbrace{\sum_{\angle(\bd_i^{(n)},\btheta)>\delta} w_i^{(n)}(\btheta)}_{=:R^{(n)}(\btheta,\delta)}.
  \label{eq:near_far}
\end{equation}
For the far sum, using $\btheta^{\!\top}\bd_i^{(n)} \leq \cos\delta < 1$ on the far indices and the softmax formula,
\begin{equation}
  R^{(n)}(\btheta,\delta)
  \leq  \frac{ne^{\cos\delta/\varepsilon^{(n)}}}{e^{1/\varepsilon^{(n)}}}
  =  n\exp\!\left(\frac{\cos\delta - 1}{\varepsilon^{(n)}}\right)
  \longrightarrow 0
  \label{eq:far_sum}
\end{equation}
as $\varepsilon^{(n)} \to 0$, since $\cos\delta - 1 < 0$ for $\delta > 0$.
Furthermore, once $h_n \leq \delta/2$ (which holds for all large $n$ since $h_n \to 0$), every $\btheta$ has at least one near index, so the near sum accounts for the dominant weight.
Inserting \eqref{eq:far_sum} into \eqref{eq:near_far} and taking $n$ large enough,
\[
  \sup_{\btheta \in \Sph^{d-1}}
  \bigl|g^{(n)}(\btheta) - g(\btheta)\bigr|
  \leq \eta + 2\|g\|_\infty \cdot n  \exp\!\left(\frac{\cos\delta-1}{\varepsilon^{(n)}}\right)  \longrightarrow \eta.
\]
Since $\eta > 0$ was arbitrary, $g^{(n)} \to g$ uniformly, hence in $L^2(\Sph^{d-1})$.

\medskip
\noindent\textbf{Step 3 — Conclusion.}
By the $L^2$ convergence $g^{(n)} \to g$ and $g \in L^2$,
\begin{align*}
  \int_{\Sph^{d-1}} f g \dd\sigma
  &=  \lim_{n \to \infty}  \int_{\Sph^{d-1}} f g^{(n)} \dd\sigma \\
  &=  \lim_{n \to \infty}  \sum_{i=1}^{n} g\!\left(\bd_i^{(n)}\right)
  \underbrace{\int_{\Sph^{d-1}} f w_i^{(n)}\dd\sigma}_{=0}
  = 0.
\end{align*}
Since this holds for every $g \in \mathcal{C}(\Sph^{d-1})$, and $\mathcal{C}(\Sph^{d-1})$ is dense in $L^2(\Sph^{d-1})$, we conclude $f = 0$ in $L^2(\Sph^{d-1})$.
\end{proof}

\begin{theorem}[Asymptotic Injectivity of the Descriptor Map]
\label{theo:injectivity}
Under the conditions of theorem~\ref{theo:completeness}, for any two points $\bp, \bq \in \M$ with $f_{\bp} \neq f_{\bq}$ in $L^2(\Sph^{d-1})$, there exist $n_0 \in \mathbb{N}$ and $i_0 \in \{1,\ldots,n_0\}$ such that
\[
  \phi_{i_0}^{(n_0)}(\bp) \neq \phi_{i_0}^{(n_0)}(\bq).
\]
In particular, $\bF^{(n_0)}(\bp) \neq \bF^{(n_0)}(\bq)$.
\end{theorem}

\begin{proof}
Suppose for contradiction that $\phi_i^{(n)}(\bp) = \phi_i^{(n)}(\bq)$ for all $i$ and all $n$.
Then $\int_{\Sph^{d-1}} (f_{\bp} - f_{\bq}) w_i^{(n)}\dd\sigma = 0$ for all $i, n$.
By theorem~\ref{theo:completeness} applied to $f = f_{\bp} - f_{\bq}$, this forces $f_{\bp} = f_{\bq}$ in $L^2$, contradicting the hypothesis.

\end{proof}

\begin{remark}
theorem~\ref{theo:completeness} implies that the proposed descriptor family is information-complete in the limit of dense directional sampling. Unlike conventional histogram descriptors, which generally discard information through finite binning, the PoU descriptor admits an asymptotically injective representation of the directional signal.
\end{remark}
 
 \subsection{Descriptor-Induced Norm on the Signal Space}
\label{sec:norming}

The descriptor map $f \mapsto \bF(f) = [\phi_1(f),\ldots,\phi_n(f)]^\top$ is a bounded linear operator from $L^2(\Sph^{d-1})$ into $\R^n$.
We show that this operator induces a natural seminorm on $L^2(\Sph^{d-1})$ via its Gram matrix \cite{CurtainZwart1995}, which becomes a genuine norm in the limit $n \to \infty$.

\begin{definition}[Gram matrix of the PoU weights]
\label{def:gram}
The \emph{Gram matrix} of the partition-of-unity weights
$\{w_i\}_{i=1}^n$ is $\mathbf{G}_n \in \R^{n \times n}$ with entries
\begin{equation}
  (\mathbf{G}_n)_{ij}
  :=  \langle w_i, w_j \rangle_{L^2(\Sph^{d-1})}
  =  \int_{\Sph^{d-1}} w_i(\btheta) w_j(\btheta) \dd\sigma(\btheta).
  \label{eq:gram}
\end{equation}
\end{definition}

\begin{proposition}[Descriptor-induced seminorm]
\label{prop:norming}
The matrix $\mathbf{G}_n$ is positive definite. Define the \emph{descriptor-induced seminorm} of $f \in L^2(\Sph^{d-1})$:
\begin{equation}
  |f|_{\mathcal{D},n}
  :=  \left( \bF(f)^{\!\top} \mathbf{G}_n^{-1} \bF(f)  \right)^{1/2}.
  \label{eq:desc_norm}
\end{equation}
Then:
\begin{enumerate}[label=\textup{(\roman*)}]
  \item\label{norm:projection}
$|f|_{\mathcal{D},n} = \|P_n f\|_{L^2(\Sph^{d-1})}$, where $P_n$ is the orthogonal projection onto $V_n := \mathrm{span}\{w_1,\ldots,w_n\} \subset L^2(\Sph^{d-1})$;
  \item\label{norm:seminorm}
$|\cdot|_{\mathcal{D},n}$ is a seminorm on $L^2(\Sph^{d-1})$ with $|f|_{\mathcal{D},n} \leq \|f\|_{L^2}$;
  \item\label{norm:limit}
$|f|_{\mathcal{D},n} \to \|f\|_{L^2}$ as $n \to \infty$, under the completeness conditions of theorem~\ref{theo:completeness}.
\end{enumerate}
\end{proposition}

\begin{proof}
\textbf{Positive definiteness of $\mathbf{G}_n$.}
For any $\boldsymbol\alpha \in \R^n \setminus \{\mathbf{0}\}$,
\[
  \boldsymbol\alpha^\top \mathbf{G}_n \boldsymbol\alpha
  =  \int_{\Sph^{d-1}}
  \Bigl(\sum_{i=1}^n \alpha_i w_i(\btheta)\Bigr)^2  \dd\sigma(\btheta)
  =  \Bigl\|\sum_{i=1}^n \alpha_i w_i\Bigr\|_{L^2}^2.
\]
This is zero only if $\sum_i \alpha_i w_i = 0$ in $L^2$, i.e.\ if $\{w_i\}$ are linearly dependent in $L^2(\Sph^{d-1})$.
For the softmax weights with distinct basis directions $\{\bd_i\}$, the functions $\{w_i\}$ are linearly independent (they have distinct peaks on $\Sph^{d-1}$), so $\mathbf{G}_n \succ 0$.

\medskip
\noindent\textbf{Proof of \ref{norm:projection}.}
The orthogonal projection $P_n f \in V_n$ satisfies the normal equations: for each $j = 1,\ldots,n$,
\[
  \langle P_n f, w_j \rangle_{L^2}
  = \langle f, w_j \rangle_{L^2}
  = \phi_j(f).
\]
Writing $P_n f = \sum_i \alpha_i w_i$, the normal equations become $\mathbf{G}_n \boldsymbol\alpha = \bF(f)$, so $\boldsymbol\alpha = \mathbf{G}_n^{-1} \bF(f)$. Therefore,
\[
  \|P_n f\|_{L^2}^2
  =  \boldsymbol\alpha^\top \mathbf{G}_n \boldsymbol\alpha
  =  \bF(f)^\top \mathbf{G}_n^{-1} \mathbf{G}_n \mathbf{G}_n^{-1} \bF(f)
  =  \bF(f)^\top \mathbf{G}_n^{-1} \bF(f)
  =  |f|_{\mathcal{D},n}^2.
\]

\medskip
\noindent\textbf{Proof of \ref{norm:seminorm}.}
Linearity of $f \mapsto \bF(f)$ and bilinearity of $\mathbf{G}_n^{-1}$ give homogeneity and the triangle inequality for $|\cdot|_{\mathcal{D},n}$.
The upper bound $|f|_{\mathcal{D},n} \leq \|f\|_{L^2}$ follows from \ref{norm:projection}: the orthogonal projection contracts, $\|P_n f\|_{L^2} \leq \|f\|_{L^2}$. 
The kernel $\ker|\cdot|_{\mathcal{D},n} = V_n^\perp$ is non-trivial for finite $n$, so $|\cdot|_{\mathcal{D},n}$ is a seminorm, not a norm.

\medskip
\noindent\textbf{Proof of \ref{norm:limit}.}
Under the hypotheses of theorem~\ref{theo:completeness}, the family $\{w_i^{(n)}\}_{i,n}$ is total in $L^2(\Sph^{d-1})$, so $V_n^{(n)} = \mathrm{span}\{w_i^{(n)}\}$ is dense:
$P_n^{(n)} f \to f$ in $L^2$ for every $f$. 
Hence $|f|_{\mathcal{D},n} = \|P_n^{(n)} f\|_{L^2} \to \|f\|_{L^2}$.

\end{proof}

\begin{corollary}[Parseval interpretation]
\label{rem:parseval}
Property \ref{norm:projection} identifies $|f|_{\mathcal{D},n}$ as the $L^2$-norm of the component of $f$ captured by the subspace $V_n$.
As $n \to \infty$, property \ref{norm:limit} asserts that the descriptor recovers the full $L^2$-norm of $f$ — a Parseval-type identity:
\[
  \lim_{n \to \infty}  \bF^{(n)}(f)^{\!\top} \bigl(\mathbf{G}_n^{(n)}\bigr)^{-1} \bF^{(n)}(f)
  = \|f\|_{L^2(\Sph^{d-1})}^2.
\]
For finite $n$, the gap $\|f\|_{L^2}^2 - |f|_{\mathcal{D},n}^2 = \|f - P_n f\|_{L^2}^2$ measures the signal energy \emph{not captured} by the descriptor — the approximation error of $V_n$ for the function $f$.
\end{corollary}

\begin{remark}[Relation to the raw descriptor norm]
\label{rem:gram_vs_euclidean}
The raw Euclidean norm $\|\bF(f)\|_2$ and the descriptor-induced norm
$|f|_{\mathcal{D},n}$ are related by
\[
  \lambda_{\min}(\mathbf{G}_n)^{1/2}|f|_{\mathcal{D},n}
  \leq   \|\bF(f)\|_2
  \leq   \lambda_{\max}(\mathbf{G}_n)^{1/2}|f|_{\mathcal{D},n}.
\]
When the weights $\{w_i\}$ are nearly orthogonal in $L^2$ (e.g.\ for small $\varepsilon$ and well-separated basis directions), $\mathbf{G}_n \approx \|w_1\|_{L^2}^2 I_n$ and the two norms are proportional, so $\|\bF(f)\|_2$ itself approximates the seminorm up to the scalar factor $\|w_1\|_{L^2}$.
\end{remark}

\begin{remark}[Induced pseudo-metric on the map]
\label{rem:pseudometric}
The descriptor-induced seminorm induces a \emph{pseudo-metric} on $\M$:
\begin{equation}
  d_n(\bp,\bq)
  :=  |f_{\bp} - f_{\bq}|_{\mathcal{D},n}
  =  \left( (\bF(\bp) - \bF(\bq))^\top  \mathbf{G}_n^{-1} (\bF(\bp) - \bF(\bq))  \right)^{1/2}.
  \label{eq:pseudo_metric}
\end{equation}
By proposition~\ref{prop:norming}\ref{norm:limit}, $d_n(\bp,\bq) \to \|f_{\bp} - f_{\bq}\|_{L^2}$ as $n \to \infty$.
Combined with the local Fréchet differentiability of proposition~\ref{prop:diff_point}, the pseudo-metric satisfies the local bi-Lipschitz bound
\[
  \sigma_{\min}(J_{\bF}(\bp))\|\bp - \bq\|_2
  \lesssim   d_n(\bp,\bq)
  \lesssim   \|J_{\bF}(\bp)\|_{\mathrm{op}}\|\bp - \bq\|_2,
\]
showing that $d_n$ metrizes the topology of $\M$ locally whenever $\sigma_{\min}(J_{\bF}(\bp)) > 0$.
\end{remark}

\subsection*{B. Stability Under Basis Perturbation}
\label{sec:stability}

\subsection{Continuity with Respect to the Directional Basis}
 
\begin{proposition}[Continuity of feature components]
\label{prop:continuity}
Let $\mathcal{D}(\tau) = \{\bd_i(\tau)\}_{i=1}^n$ depend continuously on $\tau \in \R^m$, and let $w_i(\cdot;\tau)$ be the corresponding softmax weights \eqref{eq:softmax_weight}. If $f \in L^1(\Sph^{d-1})$, then for every $i = 1, \ldots, n$ the map
\[
  \tau \longmapsto \phi_i(\bp;\tau)
  =  \int_{\Sph^{d-1}} f(\btheta) w_i(\btheta;\tau)\dd\sigma(\btheta)
\]
is continuous on $\R^m$.
\end{proposition}
 
\begin{proof}
Let $\tau_\nu \to \tau_0$.
Since $\exp(\cdot/\varepsilon)$ is continuous and $\tau \mapsto \btheta^{\!\top}\bd_i(\tau)$ is continuous for every fixed $\btheta$, we have $w_i(\btheta;\tau_\nu) \to w_i(\btheta;\tau_0)$ pointwise on $\Sph^{d-1}$.
Because $0 \leq w_i \leq 1$, the product $f(\btheta)w_i(\btheta;\tau_\nu)$ is dominated by $|f(\btheta)| \in L^1$.
The Lebesgue Dominated Convergence Theorem then gives
\[
  \lim_{\nu\to\infty}\phi_i(\bp;\tau_\nu)
  = \int_{\Sph^{d-1}} f(\btheta) \lim_{\nu\to\infty} w_i(\btheta;\tau_\nu)\dd\sigma(\btheta)
  = \phi_i(\bp;\tau_0). \qedhere
\]
\end{proof}
 
\subsection{Lipschitz Stability}

The following result quantifies how much the feature components change when the directional basis is perturbed, for a general admissible kernel $\psi$.

\begin{proposition}[Lipschitz stability of feature components]
\label{prop:lipschitz}
Let $\psi : [-1,1] \to \R_{\geq 0}$ be Lipschitz with constant $L_\psi$, and let $\mathcal{D} = \{\bd_i\}_{i=1}^n$, $\mathcal{D}' = \{\bd_i'\}_{i=1}^n$ be two directional bases.
Set
\[
  \delta := \max_{i}\|\bd_i - \bd_i'\|_2, \qquad
  Z_{\min} := \inf_{\btheta}\min\Bigl(
\textstyle\sum_j\psi(\btheta^{\!\top}\bd_j),
\textstyle\sum_j\psi(\btheta^{\!\top}\bd_j')\Bigr) > 0.
\]
Then, for every $i = 1,\ldots,n$ and every $f \in L^1(\Sph^{d-1})$,
\begin{equation}
  \bigl|\phi_i(\bp) - \phi_i'(\bp)\bigr|
  \leq  \frac{(n+1)L_\psi}{Z_{\min}}  \|f\|_{L^1}\delta,
  \label{eq:lipschitz_component}
\end{equation}
and consequently
$\|\bF(\bp) - \bF'(\bp)\|_2
 \leq \frac{(n+1)\sqrt{n}L_\psi}{Z_{\min}}\|f\|_{L^1}\delta$.
\end{proposition}

\begin{proof}
Write $\rho_i = \psi(\btheta^{\!\top}\bd_i)$, $\rho_i' = \psi(\btheta^{\!\top}\bd_i')$, $Z = \sum_j\rho_j$, $Z' = \sum_j\rho_j'$.
By Cauchy--Schwarz, $|\btheta^{\!\top}(\bd_i-\bd_i')|\leq\delta$, so the Lipschitz condition gives $|\rho_i-\rho_i'|\leq L_\psi\delta$ and $|Z-Z'|\leq nL_\psi\delta$.
Using $Z,Z'\geq Z_{\min}$ and $w_i=\rho_i/Z\leq 1$:
\[
  |w_i - w_i'|
  \leq \frac{|\rho_i-\rho_i'|}{Z'} + w_i\frac{|Z-Z'|}{Z'}
  \leq \frac{(n+1)L_\psi}{Z_{\min}}\delta =: C_\psi\delta.
\]
Integrating against $|f|$ gives \eqref{eq:lipschitz_component}; the vector bound follows from $\|\cdot\|_2\leq\sqrt{n}\|\cdot\|_\infty$.
\end{proof}

\begin{remark}[Softmax specialization and sharpness--stability trade-off]
\label{rem:tradeoff}
For the softmax kernel, $L_\psi = \varepsilon_T^{-1}e^{1/\varepsilon_T}$ and $Z_{\min} \geq ne^{-1/\varepsilon_T}$, giving
$C_{\psi_\varepsilon} = (n+1)e^{2/\varepsilon_T}/(n\varepsilon_T)$.
This constant grows as $\varepsilon_T \to 0$: sharper weights are more discriminative but more sensitive to perturbations of the basis.
\end{remark}
 
\subsection*{C. Map-Space Properties}
\label{sec:map_space}

\subsection{Stability with Respect to the Map Point}

\begin{proposition}[Continuity of feature components in $\bp$]
\label{prop:continuity_point}
Suppose the environment signal $f_{\bp} : \Sph^{d-1} \to \R$ satisfies:
\begin{enumerate}[label=\textup{(C\arabic*)}, leftmargin=2.2cm]
  \item\label{C1} for $\sigma$-almost every $\btheta \in \Sph^{d-1}$, the map $\bp \mapsto f_{\bp}(\btheta)$ is continuous on $\M$;
  \item\label{C2} there exists $g \in L^1(\Sph^{d-1})$ such that $|f_{\bp}(\btheta)| \leq g(\btheta)$ for all $\bp \in \M$ and $\sigma$-almost every $\btheta$. 
\end{enumerate}
Then for every $i = 1, \ldots, n$, the feature component $\bp \mapsto \phi_i(\bp)$ is continuous on $\M$.
Consequently, the feature vector $\bp \mapsto \bF(\bp)$ is continuous on $\M$.
\end{proposition}

\begin{proof}
Let $\bp_\nu \to \bp_0$ in $\M$. The weights $w_i(\btheta)$ are fixed (they depend only on $\mathcal{D}$ and $\varepsilon$, not on $\bp$), so
\[
  \phi_i(\bp_\nu) - \phi_i(\bp_0)
  =  \int_{\Sph^{d-1}} \bigl[f_{\bp_\nu}(\btheta) - f_{\bp_0}(\btheta)\bigr] w_i(\btheta)\dd\sigma(\btheta).
\]
By \ref{C1}, the integrand converges to zero pointwise $\sigma$-a.e.
By \ref{C2} and $0 \leq w_i \leq 1$, it is dominated by $2g(\btheta) \in L^1$.
The Dominated Convergence Theorem gives $\phi_i(\bp_\nu) \to \phi_i(\bp_0)$, and continuity of $\bF(\bp)$ follows since $n$ is finite.
\end{proof}

\begin{remark}
Condition \ref{C1} holds for the range signal when $\M$ is a smooth surface and $\bp$ moves away from surface tangencies.
Condition \ref{C2} is satisfied whenever the map is bounded. 
\end{remark}

In practice the environment signal $f_{\bp}(\btheta)$ — typically a range or occupancy measurement — is \emph{lower semicontinuous} (lsc) in $\bp$:
a new obstacle entering the line of sight causes an instantaneous \emph{decrease} in range, while the signal can only increase continuously as the obstacle recedes.
We therefore replace condition \ref{C1} by the weaker lsc assumption and show that the feature components inherit this property.

\begin{proposition}[Lower semicontinuity of feature components in $\bp$]
\label{prop:lsc_point}
Suppose the environment signal satisfies:
\begin{enumerate}[label=\textup{(L\arabic*)}, leftmargin=2.2cm]
  \item\label{L1lsc} $f_{\bp}(\btheta) \geq 0$ for all $\bp \in \M$ and $\sigma$-almost every $\btheta$;
  \item\label{L2lsc} for $\sigma$-almost every $\btheta$, the map
$\bp \mapsto f_{\bp}(\btheta)$ is lower semicontinuous on $\M$,
i.e.,
$\displaystyle\liminf_{\bp' \to \bp} f_{\bp'}(\btheta)
 \geq f_{\bp}(\btheta)$;
  \item\label{L3lsc} there exists $g \in L^1(\Sph^{d-1})$ such that
$f_{\bp}(\btheta) \leq g(\btheta)$ for all $\bp \in \M$ and $\sigma$-almost every $\btheta$.
\end{enumerate}
Then for every $i = 1, \ldots, n$, the feature component $\bp \mapsto \phi_i(\bp)$ is lower semicontinuous on $\M$.
\end{proposition}

\begin{proof}
Let $\bp_\nu \to \bp_0$ in $\M$. Since $0 \leq w_i(\btheta) \leq 1$, conditions \ref{L1lsc}--\ref{L2lsc} give
\[
  \liminf_{\nu \to \infty}  f_{\bp_\nu}(\btheta) w_i(\btheta)
  \geq  f_{\bp_0}(\btheta) w_i(\btheta)
  \qquad \sigma\text{-a.e.}
\]
The product $f_{\bp_\nu}(\btheta)w_i(\btheta)$ is non-negative and bounded above by $g(\btheta) \in L^1$ via \ref{L3lsc}.
Fatou's lemma therefore gives
\begin{align*}
  \liminf_{\nu \to \infty}\phi_i(\bp_\nu)
  &= \liminf_{\nu \to \infty}  \int_{\Sph^{d-1}} f_{\bp_\nu}(\btheta)w_i(\btheta)\dd\sigma(\btheta) \\
  &\geq  \int_{\Sph^{d-1}}  \liminf_{\nu \to \infty}  \bigl[f_{\bp_\nu}(\btheta)w_i(\btheta)\bigr]  \dd\sigma(\btheta) \\
  &\geq  \int_{\Sph^{d-1}} f_{\bp_0}(\btheta)w_i(\btheta)\dd\sigma(\btheta)
  = \phi_i(\bp_0),
\end{align*}
which is precisely the lower semicontinuity of $\phi_i$ at $\bp_0$.

\end{proof}

\begin{remark}
Condition \ref{L2lsc} captures the geometry of occlusion: as $\bp$ approaches a surface boundary, new obstacles may suddenly occlude a direction $\btheta$, causing $f_{\bp}(\btheta)$ to jump downward. The lsc assumption allows exactly such downward jumps while forbidding upward ones, which would correspond to obstacles spontaneously disappearing — a non-physical event.
Condition \ref{L3lsc} is a mild upper bound, satisfied whenever $\M$ is contained in a bounded region of $\R^d$.
\end{remark}

\subsection{Differentiability with Respect to the Map Point}
\label{sec:diff_point}

When the directional basis $\mathcal{D}$ is fixed, the weights $\{w_i\}$ do not depend on $\bp$, so differentiating $\phi_i(\bp)$ in $\bp$ reduces entirely to differentiating under the integral sign. The result gives an explicit formula for the Jacobian of $\bF(\bp)$ as a PoU-weighted average of the signal gradients.

\begin{proposition}[Fréchet differentiability of feature components]
\label{prop:diff_point}
Fix the directional basis $\mathcal{D}$ and suppose:
\begin{enumerate}[label=\textup{(D\arabic*)}, leftmargin=2.2cm]
  \item\label{D1} for $\sigma$-almost every $\btheta \in \Sph^{d-1}$, the map $\bp \mapsto f_{\bp}(\btheta)$ is differentiable on $\M$ with gradient $\nabla_{\!\bp}f_{\bp}(\btheta) \in \R^d$;
  \item\label{D2} there exists $h \in L^1(\Sph^{d-1})$ such that
$\|\nabla_{\!\bp}f_{\bp}(\btheta)\|_2 \leq h(\btheta)$ for all $\bp \in \M$ and $\sigma$-a.e.\ $\btheta$.
\end{enumerate}
Then for every $i = 1,\ldots,n$, the map $\bp \mapsto \phi_i(\bp)$ is Fréchet differentiable on $\M$ with
\begin{equation}
  \nabla_{\!\bp}\phi_i(\bp)
  =  \int_{\Sph^{d-1}} \nabla_{\!\bp}f_{\bp}(\btheta) w_i(\btheta)   \dd\sigma(\btheta)  \in \R^d.
  \label{eq:grad_component}
\end{equation}
Consequently, $\bF(\bp) \in \R^n$ is Fréchet differentiable and its Jacobian $J_{\bF}(\bp) \in \R^{n \times d}$ is
\begin{equation}
  J_{\bF}(\bp)  =   \int_{\Sph^{d-1}} \nabla_{\!\bp}f_{\bp}(\btheta)\mathbf{w}(\btheta)^{\!\top}   \dd\sigma(\btheta),
  \label{eq:jacobian}
\end{equation}
where $\mathbf{w}(\btheta) = [w_1(\btheta),\ldots,w_n(\btheta)]^\top$.
\end{proposition}

\begin{proof}
Fix $i$. Define the candidate gradient $\mathbf{g}_i(\bp) := \int \nabla_{\!\bp} f_{\bp}(\btheta)w_i(\btheta)\dd\sigma(\btheta)$, which is well defined in $\R^d$ by \ref{D2} and $0 \leq w_i \leq 1$.

We show that $\phi_i$ is Fréchet differentiable at every $\bp \in \M$ with derivative $\mathbf{g}_i(\bp)$, i.e.,
\begin{equation}
  \frac{\phi_i(\bp + \mathbf{v}) - \phi_i(\bp)   - \mathbf{g}_i(\bp)^{\!\top}\mathbf{v}}{\|\mathbf{v}\|_2}
  \longrightarrow 0
  \qquad \text{as } \mathbf{v} \to \mathbf{0}.
  \label{eq:frechet}
\end{equation}

\medskip
For $\sigma$-a.e.\ $\btheta$, condition \ref{D1} gives
\[
  f_{\bp+\mathbf{v}}(\btheta) - f_{\bp}(\btheta)
  = \nabla_{\!\bp}f_{\bp}(\btheta)^{\!\top}\mathbf{v}  + r(\btheta, \mathbf{v}),
\]
where the pointwise remainder satisfies $r(\btheta, \mathbf{v}) = o(\|\mathbf{v}\|_2)$ as $\mathbf{v} \to \mathbf{0}$, i.e., $|r(\btheta, \mathbf{v})| / \|\mathbf{v}\|_2 \to 0$.
Hence
\begin{equation}
  \phi_i(\bp+\mathbf{v}) - \phi_i(\bp) - \mathbf{g}_i(\bp)^{\!\top}\mathbf{v}
  =  \int_{\Sph^{d-1}} r(\btheta, \mathbf{v}) w_i(\btheta)\dd\sigma(\btheta).
  \label{eq:remainder_integral}
\end{equation}

\medskip
Assume $\mathbf{v}$ is small enough that the segment $\{\bp + t\mathbf{v} : t \in [0,1]\} \subset \M$.
By the multivariable mean-value inequality
\[
  |f_{\bp+\mathbf{v}}(\btheta) - f_{\bp}(\btheta)|
  \leq   \|\mathbf{v}\|_2   \sup_{t \in [0,1]}\|\nabla_{\!\bp}f_{\bp+t\mathbf{v}}(\btheta)\|_2
  \leq h(\btheta)\|\mathbf{v}\|_2,
\]
where the last inequality uses \ref{D2} applied at each $\bp + t\mathbf{v} \in \M$.
By the Cauchy--Schwarz inequality,
$|\nabla_{\!\bp}f_{\bp}(\btheta)^{\!\top}\mathbf{v}|  \leq h(\btheta)\|\mathbf{v}\|_2$.
Combining via the triangle inequality,
$|r(\btheta,\mathbf{v})| \leq 2h(\btheta)\|\mathbf{v}\|_2$, so $|r(\btheta,\mathbf{v})|/\|\mathbf{v}\|_2 \leq 2h(\btheta) \in L^1$ uniformly in $\mathbf{v} \neq \mathbf{0}$.

\medskip
Dividing \eqref{eq:remainder_integral} by $\|\mathbf{v}\|_2$ and applying the Dominated Convergence Theorem:
\[
  \frac{1}{\|\mathbf{v}\|_2}
  \int_{\Sph^{d-1}} r(\btheta, \mathbf{v}) w_i(\btheta)\dd\sigma(\btheta)
  \longrightarrow  \int_{\Sph^{d-1}} \lim_{\mathbf{v}\to\mathbf{0}}
  \frac{r(\btheta, \mathbf{v})}{\|\mathbf{v}\|_2}  w_i(\btheta)\dd\sigma(\btheta)
  = 0,
\]
which establishes \eqref{eq:frechet} and hence \eqref{eq:grad_component}.
Stacking the $n$ row vectors $\nabla_{\!\bp}\phi_i(\bp)^\top$ gives
\eqref{eq:jacobian}.
\end{proof}

\begin{remark}[Jacobian as a PoU-weighted average of signal gradients]
\label{rem:jacobian}
Equation \eqref{eq:jacobian} shows that $J_{\bF}(\bp)$ is a weighted average of the signal gradient field $\btheta \mapsto \nabla_{\!\bp}f_{\bp}(\btheta)$ over $\Sph^{d-1}$, with blending mask $\mathbf{w}(\btheta)$.
The $i$-th row $\nabla_{\!\bp}\phi_i(\bp)^\top$ captures the sensitivity of the $i$-th channel to a displacement of $\bp$: it is large when the signal changes rapidly in directions that $w_i$ weights heavily.
In particular, summing all rows and using the partition-of-unity property $\sum_i w_i = 1$ recovers the gradient of the zeroth moment:
\begin{equation}
  \mathbf{1}_n^\top J_{\bF}(\bp) 
  = \sum_{i=1}^n \nabla_{\!\bp}\phi_i(\bp)^\top
  =  \int_{\Sph^{d-1}} \nabla_{\!\bp}f_{\bp}(\btheta)\dd\sigma(\btheta)
  =  \nabla_{\!\bp} E(\bp)^\top,
  \label{eq:energy_grad}
\end{equation}
where $E(\bp) = \int f_{\bp}\dd\sigma$ is the zeroth moment of the signal.
Thus the row sum of the Jacobian is basis-free: it does not depend on $\mathcal{D}$ or $\varepsilon$.
\end{remark}

\begin{remark}[Range signal]
\label{rem:range_gradient}
For the range signal $f_{\bp}(\btheta) = r(\bp,\btheta)$ — the distance from $\bp$ to the first surface intersection in direction $\btheta$ — a ray-transfer argument yields a closed-form expression for $\nabla_{\!\bp}r$.
Let $\mathbf{x}(\bp,\btheta) = \bp + r(\bp,\btheta)\btheta$ be the hit point, lying on a smooth surface $\mathcal{S}$ with inward unit normal $\mathbf{n}(\btheta)$ at $\mathbf{x}$.
Differentiating the surface constraint $\mathbf{n}(\btheta)^{\!\top}(\mathbf{x}(\bp+\mathbf{v},\btheta) - \mathbf{x}(\bp,\btheta)) = 0$ with respect to $\mathbf{v}$ and using $\mathbf{x}(\bp+\mathbf{v},\btheta) = (\bp+\mathbf{v}) + r(\bp+\mathbf{v},\btheta)\btheta$, one obtains
\[
  \mathbf{n}(\btheta)^{\!\top}
  \bigl(\mathbf{v} + \dd r\cdot\btheta\bigr) = 0,
  \qquad\Longrightarrow\qquad
  \dd r = -\frac{\mathbf{n}(\btheta)^{\!\top}\mathbf{v}}
   {\mathbf{n}(\btheta)^{\!\top}\btheta},
\]
valid when $\mathbf{n}(\btheta)^{\!\top}\btheta \neq 0$ (non-grazing incidence).
Hence
\begin{equation}
  \nabla_{\!\bp}r(\bp,\btheta)
  =  -\frac{\mathbf{n}(\btheta)}{\mathbf{n}(\btheta)^{\!\top}\btheta},
  \label{eq:range_gradient}
\end{equation}
and substituting into \eqref{eq:grad_component} gives the closed-form expression
\begin{equation}
  \nabla_{\!\bp}\phi_i(\bp)
  =   -\int_{\Sph^{d-1}}
\frac{\mathbf{n}(\btheta)}{\mathbf{n}(\btheta)^{\!\top}\btheta}w_i(\btheta)\dd\sigma(\btheta),
  \label{eq:range_grad_component}
\end{equation}
provided the set of grazing directions $\{\btheta : \mathbf{n}(\btheta)^{\!\top}\btheta = 0\}$ has $\sigma$-measure zero, which holds generically when $\mathcal{S}$ is a smooth compact surface.
\end{remark}

\section{Map-Based Localization via Descriptor Matching}
\label{sec:localization}

We now develop the localization application. The central object is the \emph{observability Gramian}, which quantifies how much positional information the descriptor accumulates along a robot trajectory. 
In the static regime (single observation), the Gramian degenerates to a pointwise quantity and localization may fail in symmetric environments.
In the dynamic regime (robot in motion), the Gramian grows monotonically with time and generically resolves all ambiguities.

\subsection{Descriptor Grid and Static Localization}
\label{sec:static_loc}

\begin{definition}[Descriptor grid]
\label{def:grid}
Let $h > 0$.
A \emph{descriptor grid} of $\M$ at resolution $h$ is a finite set
\[
  \mathcal{G}_h
  :=  \bigl\{(\bp_k,\bF(\bp_k))\bigr\}_{k=1}^{K}   \subset \M \times \R^n,
\]
where $\{\bp_k\}_{k=1}^K$ is an $h$-covering of $\M$, i.e. $\sup_{\bp \in \M}\min_k \|\bp - \bp_k\|_2 \leq h$.
The grid is built offline by evaluating $\bF$ at each grid point.
\end{definition}

\begin{definition}[Static localization]
\label{def:static_loc}
Given $\mathcal{G}_h$ and a query descriptor $\bF_{\mathrm{obs}} \in \R^n$ observed at an unknown location $\bp^* \in \M$, the \emph{static localization problem} is
\begin{equation}
  \hat{\bp}
  =  \argmin_{\bp_k \in \mathcal{G}_h}   \bigl\|\bF(\bp_k) - \bF_{\mathrm{obs}}\bigr\|_2.
  \label{eq:static_loc}
\end{equation}
\end{definition}

The quality of \eqref{eq:static_loc} depends on how well $\bF$ separates nearby points.
Define the \emph{local injectivity radius} at $\bp^*$ as
\begin{equation}
  \mu(\bp^*)
  :=  \sigma_{\min}\!\bigl(J_{\bF}(\bp^*)\bigr)
  \geq 0,
  \label{eq:local_injectivity}
\end{equation}
where $\sigma_{\min}$ denotes the smallest singular value of the Jacobian.
When $\mu(\bp^*) > 0$, the map $\bF$ is locally injective at $\bp^*$ and the following error bound holds.

\begin{proposition}[Static localization error]
\label{prop:static_error}
Suppose $\bF \in \mathcal{C}^1(\M)$ with $\sup_{\bp \in \M}\|J_{\bF}(\bp)\|_{\mathrm{op}} \leq L$, and let $\mu = \mu(\bp^*) > 0$.
If the observation satisfies $\|\bF_{\mathrm{obs}} - \bF(\bp^*)\|_2 \leq \eta$, then for $h$ and $\eta$ small enough,
\begin{equation}
  \|\hat{\bp} - \bp^*\|_2
  \leq  \frac{L\,h + \eta}{\mu}  + o(h + \eta).
  \label{eq:static_error}
\end{equation}
\end{proposition}

\begin{proof}
By the $h$-covering property, $\|\hat\bp-\bp^*\|_2\le h$ and, using the Lipschitz bound $\|J_\bF\|_{\mathrm{op}}\le L$, $\|\bF(\hat\bp)-\bF(\bp^*)\|_2\le Lh$. By the triangle inequality,
\begin{equation}
  \|\bF(\hat\bp)-\bF(\bp^*)\|_2 \le Lh+\eta.
  \label{eq:Fgap}
\end{equation}

It remains to invert this descriptor-space bound back to position space. Since $\bF:\M\to\R^n$ maps into a higher-dimensional space ($n>d$), the classical Inverse Function Theorem does not apply directly to $\bF$; we instead apply it to a square auxiliary map built from the Moore--Penrose pseudoinverse of $J_\bF(\bp^*)$.

Because $\mu=\sigma_{\min}(J_\bF(\bp^*))>0$, the Jacobian $J_\bF(\bp^*)$ has full column rank $d$, so its Moore--Penrose pseudoinverse~\cite{BenIsraelGreville2003}
\begin{equation}
  \pi := J_\bF(\bp^*)^+ = \bigl(J_\bF(\bp^*)^\top J_\bF(\bp^*)\bigr)^{-1} J_\bF(\bp^*)^\top \in \R^{d\times n}
  \label{eq:pseudoinv}
\end{equation}
is well defined and satisfies $\pi\,J_\bF(\bp^*)=I_d$ and $\|\pi\|_{\mathrm{op}}=\mu^{-1}$.

Define the auxiliary map $\Phi:\M\to\R^d$, $\Phi(\bp):=\pi\,\bF(\bp)$. By the chain rule, $D\Phi(\bp^*)=\pi\,J_\bF(\bp^*)=I_d$, which is invertible. The classical Inverse Function Theorem therefore applies to $\Phi$ at $\bp^*$: there is a neighborhood of $\Phi(\bp^*)$ on which $\Phi$ has a $\mathcal C^1$ local inverse $\Phi^{-1}$, with
\begin{equation}
  D\Phi^{-1}\bigl(\Phi(\bp^*)\bigr) = \bigl[D\Phi(\bp^*)\bigr]^{-1} = I_d.
  \label{eq:invIFT}
\end{equation}

A first-order expansion of $\Phi^{-1}$ at $\Phi(\hat\bp)$ around $\Phi(\bp^*)$, using \eqref{eq:invIFT}, gives
\begin{equation}
  \hat\bp-\bp^*
  = \Phi^{-1}\bigl(\Phi(\hat\bp)\bigr)-\Phi^{-1}\bigl(\Phi(\bp^*)\bigr)
  = \pi\bigl(\bF(\hat\bp)-\bF(\bp^*)\bigr) + o\bigl(\|\bF(\hat\bp)-\bF(\bp^*)\|_2\bigr).
  \label{eq:taylor_inv}
\end{equation}

Taking norms in \eqref{eq:taylor_inv} and using $\|\pi\|_{\mathrm{op}}=\mu^{-1}$,
\begin{equation}
  \|\hat\bp-\bp^*\|_2
  \le \frac{1}{\mu}\,\|\bF(\hat\bp)-\bF(\bp^*)\|_2 + o\bigl(\|\bF(\hat\bp)-\bF(\bp^*)\|_2\bigr).
  \label{eq:almost_there}
\end{equation}
Substituting the descriptor-space bound \eqref{eq:Fgap} into \eqref{eq:almost_there}, and noting that $\|\bF(\hat\bp)-\bF(\bp^*)\|_2\to 0$ as $h,\eta\to 0$ so that the $o(\cdot)$ term is $o(h+\eta)$, yields
\[
  \|\hat\bp-\bp^*\|_2 \le \frac{Lh+\eta}{\mu} + o(h+\eta),
\]
which is \eqref{eq:static_error}.
\end{proof}

\begin{remark}[Failure at $\mu(\bp^*) = 0$]
The bound \eqref{eq:static_error} degenerates when $\mu(\bp^*) = 0$, i.e.\ when $J_{\bF}(\bp^*)$ has a non-trivial null space.
This means $\bF$ is locally flat along some direction from $\bp^*$: displacement along that direction leaves the descriptor unchanged.
This occurs systematically in environments with translational symmetry (corridors, circular rooms).
The dynamic regime, developed below, resolves this failure by accumulating derivative information over time.
\end{remark}
\subsection{Dynamic Localization and the Observability Gramian}
\label{sec:dynamic_loc}

Suppose the robot follows a $\mathcal{C}^1$ trajectory $\bp : [0,T] \to \M$ with velocity $\dot{\bp}(t) = \mathbf{v}(t)$ measured by onboard odometry.
By proposition~\ref{prop:diff_point}, the descriptor evolves according to the \emph{descriptor dynamics}:
\begin{equation}
  \dot{\bF}(t)
  :=  \frac{\dd}{\dd t}\bF(\bp(t))
  =  J_{\bF}(\bp(t))\mathbf{v}(t),
  \qquad t \in [0,T].
  \label{eq:descriptor_dynamics}
\end{equation}
The robot thus generates a continuous stream of observations $\{(\bF(\bp(t)),\dot{\bF}(t))\}_{t \in [0,T]}$, each constraining the unknown initial position $\bp(0)$.

\begin{definition}[Observability Gramian]
\label{def:gramian}
The \emph{observability Gramian} along the trajectory $\{\bp(t)\}_{t \in [0,T]}$ is
\begin{equation}
  \mathcal{W}(0,T)
  :=  \int_0^T J_{\bF}(\bp(t))^{\!\top} J_{\bF}(\bp(t))  \dd t
  \in \R^{d \times d}.
  \label{eq:gramian}
\end{equation}
\end{definition}

Three immediate properties of $\mathcal{W}(0,T)$:
\begin{enumerate}[label=\textup{(\roman*)}]
  \item $\mathcal{W}(0,T) \succeq 0$ for all $T \geq 0$;
  \item $T \mapsto \mathcal{W}(0,T)$ is monotone non-decreasing in the L\"{o}wner order~\cite{Bhatia1997}: each infinitesimal arc adds a rank-one update
$J_{\bF}(\bp(t))^{\!\top} J_{\bF}(\bp(t))\dd t \succeq 0$;
  \item the null space of $\mathcal{W}(0,T)$ consists of directions
$\mathbf{u}$ along which $J_{\bF}(\bp(t))\mathbf{u} = \mathbf{0}$ for $\lambda$-a.e.\ $t$ — directions permanently invisible to the descriptor along the entire trajectory.
\end{enumerate}

\begin{proposition}[Dynamic localization error bound]
\label{prop:dynamic_bound}
Let $\bp(t)$ be the true robot trajectory and $\bq(t)$ a candidate.
Assume:
\begin{enumerate}[label=\textup{(T\arabic*)}, leftmargin=2.2cm]
  \item\label{T1}\textbf{Same velocity.} $\dot{\bp}(t) = \dot{\bq}(t) = \mathbf{v}(t)$ for all $t \in [0,T]$, so that $\bdelta_0 := \bp(0) - \bq(0) = \bp(t) - \bq(t)$ is constant.
  \item\label{T2}\textbf{Bounded Jacobian.} $\sup_{t}\|J_{\bF}(\bp(t))\|_{\mathrm{op}} \leq L$.
  \item\label{T3}\textbf{Lipschitz Jacobian.} $\|J_{\bF}(\bp) - J_{\bF}(\bq)\|_{\mathrm{op}} \leq \Lambda\|\bp - \bq\|_2$ for all $\bp,\bq$ on the trajectory.
  \item\label{T4}\textbf{Observability.} $\lambda_{\min}(\mathcal{W}(0,T)) \geq \mu_T > 0$.
  \item\label{T5}\textbf{Smallness.} $\|\bdelta_0\|_2 \leq \mu_T / (\Lambda L T)$.
\end{enumerate}
Then
\begin{equation}
  \|\bp(0) - \bq(0)\|_2
  \leq  \frac{2L}{\mu_T}  \int_0^T  \bigl\|\bF(\bp(t)) - \bF(\bq(t))\bigr\|_2\dd t.
  \label{eq:dynamic_bound}
\end{equation}
In particular, $\bF(\bp(t)) = \bF(\bq(t))$ for all $t$ implies $\bp(0) = \bq(0)$.
\end{proposition}

\begin{proof}
Since $\dot{\bp}(t) = \dot{\bq}(t)$, condition \ref{T1} gives $\bq(t) = \bp(t) - \bdelta_0$ for all $t$.

\medskip
\noindent\textbf{Step 1 --- Taylor expansion with integral remainder.}
By the fundamental theorem of calculus,
\begin{equation}
  \bF(\bp(t)) - \bF(\bq(t))
  =   \int_0^1 J_{\bF}\!\bigl(\bp(t) - s\bdelta_0\bigr)  \bdelta_0\dd s
  =: J_{\bF}(\bp(t))\bdelta_0 + \mathbf{r}(t),
  \label{eq:taylor}
\end{equation}
where $\|\mathbf{r}(t)\|_2 \leq \tfrac{\Lambda}{2}\|\bdelta_0\|_2^2$ by \ref{T3}.

\medskip
\noindent\textbf{Step 2 --- Projection onto the Gramian.}
Multiplying \eqref{eq:taylor} by $J_{\bF}(\bp(t))^\top$ and integrating over $[0,T]$:
\begin{equation}
  \mathcal{W}(0,T)\bdelta_0
  =  \int_0^T J_{\bF}(\bp(t))^\top
\bigl[\bF(\bp(t)) - \bF(\bq(t))\bigr]\dd t
  -  \int_0^T J_{\bF}(\bp(t))^\top \mathbf{r}(t)\dd t.
\end{equation}
Taking norms and using \ref{T2} and \ref{T4}:
\begin{equation}
  \mu_T\|\bdelta_0\|_2
  \leq  L\int_0^T \bigl\|\bF(\bp(t)) - \bF(\bq(t))\bigr\|_2\dd t
  +  \frac{\Lambda L T}{2}\|\bdelta_0\|_2^2.
  \label{eq:norm_ineq}
\end{equation}

\medskip
\noindent\textbf{Step 3 --- Absorbing the quadratic term.}
By \ref{T5}, $\tfrac{\Lambda LT}{2}\|\bdelta_0\|_2 \leq \tfrac{\mu_T}{2}$.
Rearranging \eqref{eq:norm_ineq} gives \eqref{eq:dynamic_bound}.
\end{proof}

\begin{remark}[Information accumulation]
\label{rem:accumulation}
The bound \eqref{eq:dynamic_bound} improves with time in two ways: $\mu_T$ is non-decreasing in $T$ (more movement tightens the Gramian),
and the right-hand side averages descriptor discrepancy over $[0,T]$, suppressing sensor noise.
A longer, more informative trajectory thus yields a tighter and more noise-robust localization estimate.
\end{remark}

\subsection{Generic Resolution of Static Ambiguity}
\label{sec:symmetry}

\begin{proposition}[Generic observability]
\label{prop:generic}
For a generic map $\M$ and a generic $\mathcal{C}^1$ trajectory, $\lambda_{\min}(\mathcal{W}(0,T)) > 0$ for every $T > 0$.
\end{proposition}

\begin{proof}
$\lambda_{\min}(\mathcal{W}(0,T)) = 0$ iff there exists a unit vector $\mathbf{u} \in \R^d$ with $J_{\bF}(\bp(t))\mathbf{u} = \mathbf{0}$ for $\lambda$-a.e.\ $t$.
By \eqref{eq:grad_component}, this requires 
$\int \nabla_{\!\bp}f_{\bp(t)}(\btheta) \cdot \mathbf{u} w_i(\btheta)\dd\sigma = 0$ 
for all $i$ and a.e.\ $t$:
the directional derivative of $f_{\bp}$ in direction $\mathbf{u}$ must be annihilated by every weight $w_i$ simultaneously.
This is a non-generic condition that fails whenever $\M$ has no exact translational symmetry in direction $\mathbf{u}$.
\end{proof}

\begin{remark}[Meaning of ``generic'']
\label{rem:generic}
Fix a direction $\mathbf{u}\in\Sph^{d-1}$. Call the pair $(\M,\{\bp(t)\})$ \emph{$\mathbf{u}$-degenerate} if
\[
  J_\bF(\bp(t))\,\mathbf{u} = \mathbf 0 \qquad\text{for $\lambda$-a.e. } t\in[0,T],
\]
which, by \eqref{eq:grad_component}, is equivalent to the directional derivative of $f_{\bp(t)}$ along $\mathbf u$ being annihilated by every weight $w_i$ simultaneously, i.e.\ to $\M$ possessing an exact translational symmetry in direction $\mathbf u$ along the trajectory. For fixed $\mathbf u$, the set of maps $\M$ that are $\mathbf u$-degenerate has empty interior in the $\mathcal C^1$ topology on maps: an arbitrarily small perturbation destroys an exact symmetry, so this set is nowhere dense, and it has Lebesgue measure zero within any finite-dimensional parametric family of maps. We say the pair $(\M,\{\bp(t)\})$ is \emph{generic} if it is not $\mathbf u$-degenerate for any $\mathbf u\in\Sph^{d-1}$, i.e.\ if it avoids this nowhere-dense (measure-zero) exceptional set for every direction. Proposition~\ref{prop:generic} asserts that $\lambda_{\min}(\mathcal W(0,T))>0$ precisely for such generic pairs.
\end{remark}

\begin{remark}[Counting constraints]
\label{rem:counting}
Dynamic ambiguity requires $J_{\bF}(\bp(t))\mathbf{u} = \mathbf{0}$
for a.e.\ $t$: $nd$ scalar equations holding continuously in $t$.
For $d = 2$, $n \geq 2$ this is $2n \geq 4$ independent constraints on the $2$-dimensional ambiguity direction $\mathbf{u}$, generically yielding no solution.
\end{remark}

\subsection{Comparison and Practical Implications}

\begin{table}[h]
\centering
\caption{Static vs.\ dynamic localization.}
\label{tab:loc_summary}
\renewcommand{\arraystretch}{1.5}
\begin{tabular}{lll}
\toprule
& \textbf{Static} & \textbf{Dynamic} \\
\midrule
Observation
  & $\bF_{\mathrm{obs}}$
  & $\{\bF_{\mathrm{obs}}(t),\dot{\bF}_{\mathrm{obs}}(t)\}_{t\in[0,T]}$ \\
Discriminability
  & $\mu(\bp^*) > 0$
  & $\lambda_{\min}(\mathcal{W}(0,T)) > 0$ \\
Error bound
  & $(Lh+\eta)/\mu(\bp^*)$
  & $2L\|\text{noise}\|_{L^1[0,T]}/\mu_T$ \\
Failure condition
  & $\mu(\bp^*) = 0$ (symmetric env.)
  & $J_{\bF}\mathbf{u}=0$ a.e.\ (generically never) \\
Improves with
  & $n$ (more directions)
  & $T$ (longer trajectory) \\
Noise rejection
  & Pointwise
  & $L^1[0,T]$ averaging \\
\bottomrule
\end{tabular}
\end{table}

\noindent
In practice we recommend the following three-stage pipeline.
\begin{enumerate}[label=\textup{(\roman*)}]
  \item \textbf{Coarse initialization} via static matching \eqref{eq:static_loc} on a coarse grid.
  \item \textbf{Online refinement}: as the robot moves, monitor $\lambda_{\min}(\mathcal{W}(0,t))$; once it exceeds a threshold
$\mu_{\min}$, apply \eqref{eq:dynamic_bound} to refine the position estimate.
  \item \textbf{Active exploration}: if $\lambda_{\min}(\mathcal{W}(0,t))$ stagnates, steer toward the direction that maximizes its growth rate,
\[
  \frac{\dd}{\dd t}\lambda_{\min}(\mathcal{W}(0,t))
  =  \mathbf{u}_{\min}^\top  J_{\bF}(\bp(t))^{\!\top} J_{\bF}(\bp(t))  \mathbf{u}_{\min},
\]
where $\mathbf{u}_{\min}$ is the eigenvector of $\mathcal{W}$ corresponding to $\lambda_{\min}$.
\end{enumerate}
\section{Illustrative Example: Descriptor Computation and Localization on a 2D Map}
\label{sec:example}

We illustrate the descriptor construction of \cref{sec:formulation,sec:pou} and the static localization pipeline of \cref{sec:static_loc} on a concrete $2$-dimensional map, using an exact ray-casting sensor model, $n=8$ directions, and the softmax kernel with $\varepsilon=0.5$. All angles are in radians.

\subsection{Setup}

Consider a room $\M=[0,6]\times[0,4]$ with an internal wall segment at $x=3$, $y\in[1,3]$, and a pillar $[4.5,5]\times[1.5,2.5]$ (\Cref{fig:map}). The environment signal is the range function $f_\bp(\theta)=r(\bp,\theta)$: the distance from $\bp$ to the nearest obstacle in direction $\theta$, computed by exact ray-obstacle intersection.
The directional basis is $\theta_i=(i-1)\pi/4$, $i=1,\ldots,8$, with softmax weights \eqref{eq:softmax_weight}.

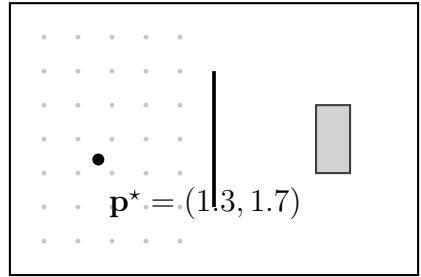
\begin{figure}[h]
\centering
\begin{tikzpicture}[scale=0.9]
  \draw[thick] (0,0) rectangle (6,4);
  \draw[thick, line width=1.4pt] (3,1) -- (3,3);
  \fill[gray!35] (4.5,1.5) rectangle (5,2.5);
  \draw[thick, gray!50!black] (4.5,1.5) rectangle (5,2.5);
  \foreach \x in {0.5,1.0,...,2.5} \foreach \y in {0.5,1.0,...,3.5}
    \fill[gray!45] (\x,\y) circle (1pt);
  \fill[black] (1.3,1.7) circle (2.5pt);
  \node[anchor=north west] at (1.3,1.45) {$\bp^\star=(1.3,1.7)$};
\end{tikzpicture}
\caption{Map with internal wall (thick segment) and pillar (shaded square). Small dots: descriptor grid $\mathcal{G}_h$ ($h=0.5$, $35$ nodes). Black dot: true query position $\bp^\star$, deliberately off-grid.}
\label{fig:map}
\end{figure}

\subsection{Convergence of the Quadrature}

The descriptor components are approximated by the discrete sum
\begin{equation}
  \hat\phi_i(\bp) = \frac{2\pi}{S}\sum_{s=0}^{S-1} r(\bp,\theta_s)w_i(\theta_s),
  \qquad \theta_s = \frac{2\pi s}{S},
  \label{eq:quad_sum}
\end{equation}
converging to \eqref{eq:component} as $S\to\infty$; the factor $2\pi/S$ is the width of each quadrature cell on $[0,2\pi)$, consistent with the definition of $\phi_i$ in \eqref{eq:component} as an unnormalized integral over $\Sph^1$. \Cref{tab:convergence} reports this convergence at $\bp=(1.5,2.0)$.

\begin{table}[h]
\centering
\caption{Convergence of $\hat\phi_i(\bp)$ as $S\to\infty$ (three representative components; error measured against $S=16384$).}
\label{tab:convergence}
\begin{tabular}{crrrr}
\toprule
$S$ & $\hat\phi_1$ & $\hat\phi_2$ & $\hat\phi_5$ & error vs.\ $S=16384$ \\
\midrule
$64$    & $1.6367$ & $1.7002$ & $1.4426$ & $5.4\times10^{-2}$ \\
$256$   & $1.5948$ & $1.6669$ & $1.4419$ & $1.2\times10^{-2}$ \\
$1024$  & $1.5850$ & $1.6589$ & $1.4415$ & $2.2\times10^{-3}$ \\
$4096$  & $1.5826$ & $1.6570$ & $1.4414$ & $2.0\times10^{-4}$ \\
$16384$ & $1.5828$ & $1.6571$ & $1.4414$ & --- (reference) \\
\bottomrule
\end{tabular}
\end{table}

Convergence is $O(1/S)$ on average rather than spectral, because $r(\bp,\cdot)$ has jump discontinuities: at $\theta=0.588$ rad the ray from $\bp$ grazes the top edge of the internal wall ($y=3$), and $r$ jumps from $1.80$ to $3.61$ (a symmetric jump occurs at $\theta=2\pi-0.588$, the bottom edge $y=1$). These shadow-boundary singularities are averaged, but not removed, by the smooth PoU weights, which caps the quadrature order at $O(1/S)$. The observed error ratio between successive doublings of $S$ fluctuates around the value $2$ predicted by $O(1/S)$ (empirically between $1.0$ and $3.7$ across the range $S=64,\ldots,16384$) rather than decreasing monotonically, because the fixed discontinuity locations $\theta=0.588,\,2\pi-0.588$ land at varying distances from the nearest quadrature node as $S$ changes. For the remainder of this example we use $S=1024$, accurate to $\approx2\times10^{-3}$ per component.

\subsection{Localization at an Off-Grid Query Point}

A sensor at unknown position $\bp^\star=(1.3,1.7)$ (not a grid node; see \Cref{fig:map}) returns the observed descriptor $\bF_{\mathrm{obs}}=\bF(\bp^\star)$. The localizer has access only to $\bF_{\mathrm{obs}}$ and the precomputed grid $\{(\bp_k,\bF(\bp_k))\}_{k=1}^{35}$; it does not know $\bp^\star$.

\paragraph{Stage 1: nearest-neighbor search \eqref{eq:static_loc}.} \Cref{tab:nn} lists the four nearest grid nodes.

\begin{table}[h]
\centering
\caption{Top-4 nearest grid nodes to $\bF_{\mathrm{obs}}$.}
\label{tab:nn}
\begin{tabular}{cccc}
\toprule
Rank & $\bp_k$ & $\|\bF(\bp_k)-\bF_{\mathrm{obs}}\|_2$ & true error $\|\bp_k-\bp^\star\|_2$ \\
\midrule
1 & $(1.5,1.5)$ & $0.2869$ & $0.283$ \\
2 & $(1.5,2.0)$ & $0.3839$ & $0.361$ \\
3 & $(1.0,2.0)$ & $0.4709$ & $0.424$ \\
4 & $(1.0,1.5)$ & $0.4757$ & $0.361$ \\
\bottomrule
\end{tabular}
\end{table}

The nearest node $\hat\bp^{(0)}=(1.5,1.5)$ localizes $\bp^\star$ to within one grid cell, consistent with the static error bound \eqref{eq:static_error}; grid resolution limits stage 1 to $O(h)$ accuracy.

\paragraph{Stage 2: Gauss--Newton refinement~\cite{NocedalWright2006}.} Using the Jacobian $J_\bF(\bp)\in\R^{8\times2}$ (central differences, $h=0.05$), iterate
\begin{equation}
  \bp^{(k+1)} = \bp^{(k)} + \bigl(J_\bF^\top J_\bF\bigr)^{-1} J_\bF^\top\bigl(\bF_{\mathrm{obs}}-\bF(\bp^{(k)})\bigr),
  \label{eq:gn_iterate}
\end{equation}
evaluated at $\bp^{(k)}$.
\Cref{tab:gn} shows rapid convergence.

\begin{table}[h]
\centering
\caption{Gauss--Newton refinement from $\hat\bp^{(0)}=(1.5,1.5)$.}
\label{tab:gn}
\begin{tabular}{clcc}
\toprule
Iter. & Position estimate & $\|\bF(\bp^{(k)})-\bF_{\mathrm{obs}}\|_2$ & error to $\bp^\star$ \\
\midrule
$0$ & $(1.500,1.500)$ & $0.2869$ & $0.2828$ \\
$1$ & $(1.289,1.694)$ & $0.0165$ & $0.0124$ \\
$2$ & $(1.301,1.700)$ & $0.0008$ & $0.0007$ \\
$3$ & $(1.300,1.700)$ & $0.0002$ & $0.0001$ \\
\bottomrule
\end{tabular}
\end{table}

Three Gauss--Newton steps reduce the localization error from $0.283$ (grid resolution) to $10^{-4}$, recovering $\bp^\star=(1.3,1.7)$ to numerical precision; the residual $\|\bF(\bp^{(k)})-\bF_{\mathrm{obs}}\|_2$ decreases monotonically at each step, as expected from a well-posed Gauss--Newton iteration.

\paragraph{Certificate of well-posedness.} At $\hat\bp^{(0)}$, the Jacobian singular values are $\sigma=(1.452,0.845)$, so the local injectivity radius \eqref{eq:local_injectivity} is
\[
  \mu(\hat\bp^{(0)}) = \sigma_{\min}(J_\bF) = 0.845 > 0,
\]
certifying, via Proposition~\ref{prop:static_error}, that $\bF$ is locally injective and the Gauss--Newton iteration is well-conditioned. A vanishing $\mu$ (e.g.\ in a long symmetric corridor, cf.\ \cref{sec:symmetry}) would leave stage 2 unable to resolve position along the flat direction, and only the coarser grid bound of stage 1 would apply.

\subsection{Summary}

\begin{table}[h]
\centering
\begin{tabular}{ll}
\toprule
Quadrature accuracy & $S=1024$ ($\approx2\times10^{-3}$ per component) \\
Grid & $35$ nodes, $h=0.5$ \\
True position & $\bp^\star=(1.3,1.7)$ (off-grid) \\
Stage-1 (nearest-neighbor) error & $0.283$ \\
Stage-2 (Gauss--Newton) error & $1.4\times10^{-4}$ \\
Injectivity certificate & $\mu(\hat\bp^{(0)})=0.845>0$ \\
\bottomrule
\end{tabular}
\end{table}

\noindent
In summary, localization proceeds in two stages: a coarse nearest-neighbor search on the precomputed grid, accurate to $O(h)$, followed by Gauss--Newton refinement using the descriptor Jacobian, accurate to numerical precision once inside the basin of convergence, and certified throughout by the local injectivity radius $\mu>0$.

\section{Acknowledgments*}

\bibliographystyle{unsrt}
\bibliography{references}

\end{document}